\documentclass[12pt,draftcls,onecolumn,letter]{IEEEtran}

\IEEEoverridecommandlockouts

\usepackage{lineno}
\usepackage{graphicx}
\usepackage{epsfig}
\usepackage{subfig}
\usepackage{times}
\usepackage{amsmath}
\usepackage{amssymb}
\usepackage{hyperref}

\usepackage{cite}
\usepackage{color}

\newtheorem{definition}{Definition}[section]

\newtheorem{remark}{Remark}[section]
\newtheorem{lemma}{Lemma}[section]
\newtheorem{theorem}{Theorem}[section]

\title{\Large \bf Further clarifications of ``Necessary and sufficient stability condition of fractional-order interval linear systems''}

\author{\small Hyo-Sung Ahn$^\dag$, Young-Hun Lim$^\dag$, Kwang-Kyo Oh$^\dag$, and YangQuan Chen$^\ddag$
\thanks{\small $^\dag$School of Mechatronics, Gwangju Institute of Science and Technology, 123 Cheomdangwagi-ro, Buk-gu, Gwangju, 500-712 South Korea.
{E-mail: hyosung@gist.ac.kr}} \\
\thanks{\small $^\ddag$University of California, Merced, 5200 North Lake Road, Merced, CA 95343, USA.}}

\begin{document}

\maketitle

\begin{abstract}
This technical report replies to the comments of \cite{Comment_2014} in detail, and corrects a possible mis-interpretation of \cite{Ahn_Auto2008} in terms of the conventional robust stability concept. After defining the robust stability and quadratic stability concepts, it will be more clearly described that the condition established in \cite{Ahn_Auto2008} can be considered an exact condition for the quadratic stability of fractional-order interval systems. 
\end{abstract}

\section{Introduction} 
First of all, we would like to appreciate the authors of \cite{Comment_2014} for giving us a chance to review the results of \cite{Ahn_Auto2008}. This technical report corrects a potential mis-interpretation of the works in \cite{Ahn_Auto2008} in terms of the conventional robust stability. 
After having a careful thinking on the results of \cite{Ahn_Auto2008} motivated by the comments \cite{Comment_2014}, we have found that the inaccurate statements and terminologies of \cite{Ahn_Auto2008} may induce some confusions to readers. Furthermore, we have found that there could be a mis-interpretation of the \textit{Lemma~$2$} of \cite{Ahn_Auto2008} in terms of the conventional robust stability concept. Even though the works of \cite{Ahn_Auto2008} should have been understood in terms of quadratic stability, since it is highly possible to mislead the readers, we would like to provide detailed clarifications in this technical report.

Firstly, to clarify the results of \cite{Ahn_Auto2008} in a clear way, we would like to provide the concepts of robust stability and quadratic stability in fractional-oder interval systems. Then, based on the definitions, we will clearly show that the results of \cite{Ahn_Auto2008} can be considered as an exact condition in terms of quadratic stability, but it is a sufficient condition in terms of the conventional robust stability. 



\section{Notations and Definitions} \label{sec2}
A matrix $A$ is denoted as $A \doteq [a_{ij}]$ and an interval parameter is denoted as $a_{ij} \in a_{ij}^I = \lfloor \underline{a_{ij}}, \overline{a_{ij}} \rfloor$, where $\underline{a_{ij}}$ and $\overline{a_{ij}}$ are lower and upper boundaries of the interval parameter. Consider the following linear time-invariant (LTI) fractional-order interval systems 
\begin{eqnarray} \label{eq1}
\frac{{\text d}^\alpha x(t) }{{\text d}t^\alpha} = A x, ~ A \in {\mathcal A}^{I}
\end{eqnarray}
where $1 \leq \alpha <2$, $A \doteq [a_{ij}]$, $ a_{ij} \in \lfloor \underline{a_{ij}}, \overline{a_{ij}} \rfloor$ and the interval matrix set is defined as ${\mathcal A}^I \doteq [ \lfloor \underline{a_{ij}}, \overline{a_{ij}} \rfloor ]$. Define the set of vertex matrices as ${\mathcal A}^v \doteq [ \{  \underline{a_{ij}}, \overline{a_{ij}} \}]$. It is noticeable that the interval matrix set ${\mathcal A}^I$ and vertex matrix  set ${\mathcal A}^v$ are sets with the relationship of ${\mathcal A}^v \subset {\mathcal A}^I$. 
Within the set ${\mathcal A}^I$, there are infinitely many element matrices, while in the set ${\mathcal A}^v$, there are finite number of element matrices. Let individual element matrices $A$ within the set ${\mathcal A}^I$ denote as 
$A_i^{int}$, i.e., $A_i^{int} \in {\mathcal A}^I$, where we use superscript `$int$' to denote the individual element matrices within interval, and the subscript `$i$' to denote $i$-th matrix. So, we have $i \in {\mathcal I}^{index} \doteq \{1,2,3, \cdots, \infty \}$. Let individual element matrices $A$ within the set ${\mathcal A}^v$ denote as 
$A_i^{ver}$, i.e., $A_i^{ver} \in {\mathcal A}^v$, where we use superscript `$ver$' to denote the individual vertex matrices, and the subscript `$i$' to denote $i$-th vertex matrix; in this case, $i \in {\mathcal V}^{index} \doteq \{1,2,3, \cdots, N\}$, where $N$ is the cardinality of the vertex matrix set.  For a Hermitian matrix $A$, the conjugate transpose is expressed as $A^\ast$, and if it is positive definite, then it is written as $A=A^\ast >0$ (we use $A=A^\ast <0$ for negative definite). Now, we provide the following two concepts for stability of fractional-order interval systems. 

\begin{definition} \label{def_rs}
(Robust stability) The fractional-order interval systems (\ref{eq1}) would be robust stable if there exist $P_i^{int}=(P_i^{int})^\ast>0$ in accordance with $A_i^{int}$ such that 
\begin{eqnarray}
\beta P_i^{int} A_i^{int} + \beta^\ast (A_i^{int})^T P_i^{int} <0, \forall A_i^{int} \in {\mathcal A}^I \label{eq_rs}
\end{eqnarray}
\end{definition}

\begin{definition} \label{def_qs}
(Quadratic stability) Based on \cite{Hideki_tac1993,Amato2006},  the fractional-order interval systems (\ref{eq1}) shall be said to be quadratically stable if there exists a $P=P^\ast>0$ for all $A_i^{int}$ such that 
\begin{eqnarray}
\beta P A_i^{int} + \beta^\ast (A_i^{int})^T P <0, \forall A_i^{int} \in {\mathcal A}^I  \label{eq_qs}
\end{eqnarray} 
\end{definition}

Note that $P_i^{int}$ in (\ref{eq_rs}) could be different according to the selection of $A_i^{int} \in {\mathcal A}^I$, while in (\ref{eq_qs}), $P$ is fixed as a common matrix regardless of $A_i^{int}$. Also note that in \cite{Hideki_tac1993,Amato2006}, the linear time-variant (LTV) cases are considered for quadratic stability; but for a clarity of description and for the consistence of presentation, we consider only the LTI fractional-order interval cases. But, the results of \cite{Ahn_Auto2008} would result in an exact condition for the quadratic stability in LTV cases also. A similar result for integer-order LTV cases was established in \cite{BYKim2011}. 

\begin{remark}
In the sense of the \textit{Definition~\ref{def_qs}}, the \textit{Lemma~1} of \cite{Ahn_Auto2008} could be (or should have been) understood as a kind of definition for quadratic stability of fractional-order interval systems. See also the \textit{Lemma~\ref{lemma1-qs}} in the next section.
\end{remark}

\section{Reply to \cite{Comment_2014} and Corrections of \cite{Ahn_Auto2008}} \label{sec3}
In \cite{Comment_2014}, it is pointed out that the \textit{Lemma~$1$} and \textit{Theorem~$3$} of \cite{Ahn_Auto2008} are not true; but they only provide sufficient conditions. 
As aforementioned, the inaccurate use of words (i.e., robust stability) in \textit{Lemma~$1$} of \cite{Ahn_Auto2008} might have been misleading the readers. 
In what follows, we would like to reply to the comments of \cite{Comment_2014} by way of clarifying the results of \cite{Ahn_Auto2008} in a much more clear way.

First of all, to avoid a confusion arising from the statement of the \textit{Lemma~$1$} of \cite{Ahn_Auto2008}, based on the \textit{Definition~\ref{def_rs}} and \textit{Definition~\ref{def_qs}}, it might be necessary to 
rewrite it into the following lemmas:
\begin{lemma} \label{lemma1-rs}
The system (\ref{eq1}) is robust stable if and only if, for individual matrices $A_i^{int} \in {\mathcal A}^I$, there exist matrices $P_i^{int} = (P_i^{int})^\ast >0$ such that 
\begin{eqnarray} \label{eq2}
\max\{ \overline{\lambda_i} (\beta P_i^{int} A_i^{int} + \beta^\ast (A_i^{int})^T P_i^{int})  \} <0 
\end{eqnarray}
where $\overline{\lambda_i}$ is the maximum eigenvalue of $\beta P_i^{int} A_i^{int} + \beta^\ast(A_i^{int})^T P_i^{int}$.
\end{lemma}  
\begin{proof}
The proof is direct from the robust stability concept defined in the \textit{Definition~\ref{def_rs}} and from the arguments given in the first paragraph of Section~$2$ of \cite{Ahn_Auto2008}. 
\end{proof}

\begin{lemma} \label{lemma1-qs}
The system (\ref{eq1}) is quadratically stable if and only if, for all interval element matrices $A_i^{int} \in {\mathcal A}^I$, there exists a $P = P^\ast >0$ such that 
\begin{eqnarray} \label{eq2}
\max\{ \overline{\lambda_i} (\beta P A_i^{int} + \beta^\ast (A_i^{int})^T P)  \} <0 
\end{eqnarray}
where $\overline{\lambda_i}$ is the maximum eigenvalue of $\beta P A_i^{int} + \beta^\ast(A_i^{int})^T P$.
\end{lemma}  
\begin{proof}
The proof is direct from the quadratic stability concept defined in the \textit{Definition~\ref{def_qs}}. 
\end{proof}
\begin{remark}
Actually, the \textit{Lemma~1} in \cite{Ahn_Auto2008} corresponds to the above \textit{Lemma~\ref{lemma1-qs}} (not \textit{Lemma~\ref{lemma1-rs}}). Also note that the result for the integer-order cases of the \textit{Lemma~\ref{lemma1-rs}} is given in \cite{Kolev_tac2005}. 
\end{remark}
Obviously, it is impossible to check the condition given in the above lemmas since they are dependent on all interval element matrices. 
To handle this problem, the authors of \cite{Ahn_Auto2008} used an algebraic relationship between the maximum eigenvalue and quadratic maximization. That is, to check the stability by using only a finite number of element matrices, which is the main goal of \cite{Ahn_Auto2008}, 
they defined $\overline{\lambda}$ for any Hermitian matrix $P$ with appropriate dimension (see eq. (3) of \cite{Ahn_Auto2008}):
\begin{eqnarray} \label{eq3}
\overline{\lambda}(P) & \doteq &    \max_{A_i^{int} \in {\mathcal A}^I} \left\{ \overline{\lambda_i} (\beta P A_i^{int} + \beta^\ast (A_i^{int})^T P)  \right \} \nonumber\\
          & = &\max_{A_i^{int} \in {\mathcal A}^I} \left \{ \max_{ \Vert z \Vert =1} z^\ast (\beta P A_i^{int} + \beta^\ast (A_i^{int})^T P) z \right \}~~~
\end{eqnarray}
 
After some algebraic manipulations, based on (\ref{eq3}), the authors of \cite{Ahn_Auto2008} obtained the following lemma (see \textit{Lemma~$2$} of \cite{Ahn_Auto2008}):
\begin{lemma} \label{lemma2} 
For any Hermitian matrix $P$ with appropriate dimension, the value of \eqref{eq3} is maximized at one of the vertex matrices. 
\end{lemma}

%

The above \textit{Lemma~\ref{lemma2}} can be interpreted in two ways. 
First, if there exists a Hermitian matrix $P^{common}$ that renders the value of $\overline{\lambda}(P^{common})$ negative, we have the following inequality for all $A_i^{int} \in {\mathcal A}^I$: 
\begin{align*}
\max\{ \overline{\lambda_i} (\beta P^{common} A_i^{int} + \beta^\ast (A_i^{int})^T P^{common})  \} <0,
\end{align*}
which leads a robust stability condition of the system \eqref{eq1}. This can be summarized as follows:
\begin{theorem} \label{new_thm_rs}
The system (\ref{eq1}) is robust stable if there exists a common positive-definite matrix  $P^{common} = (P^{common})^\ast$  for all vertex matrices $A_i^{ver} \in {\mathcal A}^v$ such that 
\begin{eqnarray} \label{eq4_rs}
&& \overline{\lambda_i} (\beta P^{common} A_i^{ver} + \beta^\ast (A_i^{ver})^T P^{common})  <0,  \forall  A_i^{ver} \in   {\mathcal A}^v
\end{eqnarray}
\end{theorem}
\begin{proof}
The proof is direct from the \textit{Definition~\ref{def_rs}} and from the \textit{Lemma~\ref{lemma1-rs}}. 
\end{proof}

Second, when applying the quadratic stability concept given in the \textit{Definition~\ref{def_qs}}, the \textit{Lemma~\ref{lemma2}} can be considered as an exact condition, which is summarized in the following theorem:
\begin{theorem} \label{new_thm_qs}
The system (\ref{eq1}) is quadratically stable if and only if there exists a common positive-definite matrix  $P^{common} = (P^{common})^\ast$  for all vertex matrices $A_i^{ver} \in {\mathcal A}^v$ such that 
\begin{eqnarray} \label{eq4_qs}
&& \overline{\lambda_i} (\beta P^{common} A_i^{ver} + \beta^\ast (A_i^{ver})^T P^{common})  <0, \forall  A_i^{ver} \in   {\mathcal A}^v
\end{eqnarray}
\end{theorem}
Though the proof of this theorem is eventually same to the proof of the \textit{Theorem 3} of \cite{Ahn_Auto2008}, we would like to clarify it again in the following proof. 

\begin{proof}
The \textit{only if ($\Longrightarrow$)} condition is clear from the context. The \textit{if  ($\Longleftarrow$)} condition can be investigated as follows. From the \textit{Lemma~\ref{lemma2}}, for a fixed $P^{common}>0$, we have 
\begin{align}
 \max_{A_i^{int} \in {\mathcal A}^I} & \{ \max_{ \Vert z \Vert =1} z^\ast (\beta P^{common} A_i^{int} + \beta^\ast (A_i^{int})^T P^{common}) z  \} \nonumber\\
 \leq &\max_{A_i^{ver} \in {\mathcal A}^v}  \{ \max_{ \Vert z \Vert =1} z^\ast (\beta P^{common} A_i^{ver} + \beta^\ast (A_i^{ver})^T P^{common}) z \}. 
\end{align}

Also from the \textit{Lemma~\ref{lemma2}}, it is inferred that if there exists a fixed $P^{common}$ satisfying $\max_{A_i^{ver} \in {\mathcal A}^v}  \{ \max_{ \Vert z \Vert =1} z^\ast (\beta P^{common} A_i^{ver}$ $+ \beta^\ast (A_i^{ver})^T P^{common}) z \}<0$, then the same $P^{common}$ exists ensuring $\max_{A_i^{int} \in {\mathcal A}^I}  \{ \max_{ \Vert z \Vert =1}$ $ z^\ast (\beta P^{common} A_i^{int} + \beta^\ast $ $(A_i^{int})^T  P^{common}) z  \} <0$. Thus, if there exists a common $P^{common} = (P^{common})^\ast>0$ satisfying (\ref{eq4_qs}), there will exist a common $P^{common} = (P^{common})^\ast>0$ such that 
(\ref{eq_qs}) holds. 
\end{proof}


It should be also noticed that the example-1 of \cite{Ahn_Auto2008} is quadratic stable because there exists $P^{common} = (P^{common})^\ast$ such as:
\begin{eqnarray} 
P^{common} = \left[ \begin{array}{ccc} 0.8575 & 0.1313 + j 0.1332 & 0.1613 + j 0.3652\\
         0.1313 - j 0.1332 & 0.7062 & -0.0051 + j 0.5039\\ 0.1613 - j0.3652 & -0.0051 - j 0.5039 & 1.0618 
\end{array}
\right]
\end{eqnarray}

\section{Conclusions and Remarks}
Motivated by the comments of \cite{Comment_2014}, we have checked correctness of the results of \cite{Ahn_Auto2008}. By using the concepts of robust stability and quadratic stability, we have clarified that the results of \cite{Ahn_Auto2008} would be considered a sufficient one in terms of robust stability, while it is an exact necessary and sufficient condition in terms of quadratic stability. 
%

As some possible future works, the following remark may be useful. 
\begin{remark}
It is noticeable that even in integer-order LTI interval systems, there has been no necessary and sufficient condition for robust stability. Thus, as far as the authors are concerned, it is still an open problem to find an exact and algebraically feasible solution for the robust stability of fractional-order interval systems. But, the discussions of this correspondence may provide some ideas on how to find an exact solution. Since the set of vertex matrices needs to be stable  and it also provides a sufficient condition (i.e., the existence of common quadratic Lyapunov function), it looks like that the vertex matrices still might play a key role in establishing the exact robust stability condition for some classes of fractional-order interval systems. Related to this argument, better results might be developed for fractional-order cases using the results in integer-order systems recently developed in \cite{Kolev_ijcta2010,Pastravanu_tac2011,Firouzbahrami_cta2013}. 
\end{remark}

\section{Acknowledgements}
The authors of \cite{Ahn_Auto2008} would like to once again thank the authors of \cite{Comment_2014} for giving us a chance to clarify the results of \cite{Ahn_Auto2008} in a clear way.

\end{document}